\documentclass[opre,nonblindrev]{informs3_hide}
\OneAndAHalfSpacedXI

\def\theARTICLETOP{}
\ARTICLERULE{}

\usepackage[hidelinks,hypertexnames=false]{hyperref}
\usepackage{mathrsfs,mathtools,amsfonts,bm}
\usepackage{needspace}
\usepackage{float}
\usepackage{tikz}
\usetikzlibrary{arrows.meta,positioning}
\usepackage{natbib}
\bibpunct[, ]{(}{)}{,}{a}{}{,}
\def\bibfont{\small}

\TheoremsNumberedThrough
\ECRepeatTheorems
\EquationsNumberedThrough
\allowdisplaybreaks[4]

\begin{document}

\newcommand{\E}{\mathbb{E}}
\newcommand{\Prob}{\mathbb{P}}
\newcommand{\Z}{\mathbb{Z}}
\newcommand{\ind}{\mathbf{1}}
\newcommand{\eqdist}{\overset{d}{=}}

\TITLE{On Niu's Conjecture for Two-Station Tandem Queues}
\RUNTITLE{Niu's Conjecture}
\RUNAUTHOR{Cao, Hu, and Wang}

\ARTICLEAUTHORS{
\AUTHOR{Hao Cao}
\AFF{School of Management, Fudan University, 
\EMAIL{hcao21@m.fudan.edu.cn}}
\AUTHOR{Jian-Qiang Hu}
\AFF{School of Management, Fudan University, 
\EMAIL{hujq@fudan.edu.cn}}
\AUTHOR{Qixin Wang}
\AFF{School of Management, Fudan University, 
\EMAIL{23110690014@m.fudan.edu.cn}}
}

\ABSTRACT{
In this paper, we prove a long standing conjecture in queueing theory made by \cite{Niu1980}: in a two-station tandem queueing system with constant service time ($c$) at station 1, the stationary average waiting time at station 2 is nonincreasing in \(c\).  Two very different proofs are provided: one for the general case and one for a special case in which the arrival process is Poisson and the service time at station 2 is exponential.
}

\maketitle

\section{Introduction}\label{sec:model}
In 1980, \cite{Niu1980} made the following strong conjecture:
\begin{quote}
    {\bf In a two-station tandem queueing system with constant service time ($c$) at station 1 ($GI/D/1 \rightarrow G/1$), the stationary average waiting time at station 2 is nonincreasing in \(c\).}
\end{quote}
Previous results give bounds on the downstream mean waiting time (\citealp{tembewolff1974,Niu1980}), light and heavy traffic approximations (\citealp{wolff1982,harrison1978,glynnwang2023}), and comparisons in light traffic or a special heavy traffic limit (\citealp{greenberg1987}).
Despite these efforts, the conjecture remains open.  Our work in this paper is the first one to show the conjecture is true. We provide two very different proofs: one for the general case and one for the special case with Poisson arrivals and exponential service times at station 2. The rest of the paper is organized as follows.  Section 2 provides some preliminary results.  Section 3 considers the general case, and Section 4 considers the special case. Section 5 concludes the paper. Some technical proofs are provided in the appendix.  

\section{Preliminaries}
We consider two first-come, first-served single-server and unlimited buffer queues in tandem, with renewal arrivals, a constant service time at the first station, and independent, identically distributed (IID) service times at the second.  Customers arrive at station 1 in the order of $\ldots, -3, -2, -1, 0, 1, 2, 3, \ldots$. Let $\tau_k$ be the arrival time of customer $k$ at station 1, and set $\tau_0=0$. Let $A_k=\tau_k-\tau_{k-1}$ be the inter-arrival time between customers $k-1$ and $k$, $c$ the constant service time at station 1, $S_k$ the service time of customer $k$ at station 2. $(A_k)_{k\in\Z}$ and $(S_k)_{k\in\Z}$ are IID and independent of one another.  We use $A$ to denote a generic inter-arrival time and $S$ a generic service time at station 2.  We assume that $A,S,c\geq0$ and $\max\{c,\E[S]\}<\E[A]<\infty$ to ensure the stability of the queueing system. Simultaneous events are processed in customer order.   Let $D_{j,k}$ be the departure time of customer $k$ from station $j$, $W_{j,k}$ the waiting time of customer $k$ at station $j$, and $W_k = W_{1,k}+ W_{2,k}$.  Note that \(D_{j,k}\) and \(W_{j,k}\) are all dependent on \(c\); we may write them as \(D_{j,k}(c)\) and \(W_{j,k}(c)\) if necessary.

Consider a sample path starting with customer \(-N\) and assume that both stations are empty when its first customer \(-N\) arrives.  In our analysis, we will focus on customer~0. Write \(Y_i:=c-A_i\) for \(-N+1\leq i\leq0\) and \(\Xi_m:=\sum_{i=m}^{-1}S_i\) for integers \(m\leq0\), with \(\Xi_0=0\). Then the waiting times of customer \(-N\leq k\leq0\) can be expressed as
\begin{align*}
W_{1,k}&=\max_{-N\le m\le k}\left\{ \sum_{i=m+1}^k(c-A_i) \right\}
=\max_{-N\le m\le k} \left\{ \sum_{i=m+1}^k Y_i\right\},\\
W_k&=\max\nolimits_{-N\le m\le n\le k}\big\{\sum\nolimits_{i=n}^{k-1}S_i+(n-m)c-\sum\nolimits_{i=m+1}^k A_i\big\}\\
&=\max\nolimits_{-N\leq n\leq k}\big\{W_{1,n}-\sum\nolimits_{i=n+1}^k A_i+\sum\nolimits_{i=n}^{k-1}S_i\big\},
\end{align*}
and \(W_{2,k}=W_k-W_{1,k}\). Define \(Z_m:=\max_{m\leq n\leq0}\{\Xi_n+nc\}\), \(\Delta Z_i:=Z_{i-1}-Z_i\), \(U_i:=Y_i+\Delta Z_i\), where \(-N\leq m\leq0\), \(-N+1\leq i\leq0\), and \(Z_0=0\). Then we have
\begin{align*}
W_0 & = \max_{-N\le m\le n\le 0} \left \{\Xi_{n}+(n-m)c-\sum_{i=m+1}^0 A_i \right \}  = \max_{-N\leq m\leq 0} \left\{ \sum_{i=m+1}^0 U_i\right\}.
\end{align*}
Let \(b_j(k)\) denote the index of the first customer in the busy period containing customer \(k\) at station \(j\) (\(j=1,2\)). Then
\begin{align*}
 b_1(k)&=\argmax\nolimits_{-N\leq m\leq k}\big\{\sum\nolimits_{i=m+1}^{k}Y_i\big\},\\
 b_2(0)&=\argmax\nolimits_{-N\leq n\leq0}\big\{W_{1,n}-\sum\nolimits_{i=n+1}^{0}A_i+\Xi_n\big\},\\
 b_1(b_2(0))&=\argmax\nolimits_{-N\leq m\leq0}\big\{\sum\nolimits_{i=m+1}^{0}U_i\big\}.
\end{align*}
In the above definitions, we have implicitly assumed that the maximizers are almost surely unique for each fixed \(c\). This assumption holds under some very mild conditions, e.g., when \(A\) and \(S\) are continuous random variables. We make this assumption mainly for simplicity of notation in our derivation. Using the technique of infinitesimal perturbation analysis (IPA), we can obtain the sample path derivatives of \(W_{1,0}\), \(W_{2,0}\), and \(W_0\) (with respect to \(c\)) as follows (\citealt{hu1992}):
\begin{align}\label{eq_IPA}
 \frac{dW_{1,0}\left(c\right)}{dc}&=-b_1\left(0\right),~
 \frac{dW_0\left(c\right)}{dc}=b_2\left(0\right)-b_1\left(b_2\left(0\right)\right),~
 \frac{dW_{2,0}\left(c\right)}{dc}=b_2\left(0\right)-b_1\left(b_2\left(0\right)\right)+b_1\left(0\right).
\end{align}
For the tandem queueing system considered here, the sample path derivatives are both unbiased for each finite $N$ and strongly consistent in steady state (as $N\rightarrow \infty$) (\citealt{hu1992,ghs1991}).  Therefore, to prove Niu's conjecture, it suffices to show
\begin{align}\label{eq_mainresult}
 \E\left[\frac{dW_{2,0}\left(c\right)}{dc}\right]
 &=\E\left[b_2\left(0\right)-b_1\left(b_2\left(0\right)\right)+b_1\left(0\right)\right]\leq0,
\end{align}
for each finite $N$.

Before closing this section, we present the main result of our paper:
\begin{theorem}\label{thm:mean}
(\ref{eq_mainresult}) holds in steady state, i.e., Niu's conjecture is true. 
\end{theorem}

\section{General case} \label{sec:main results}
Fix \(c\) and the second service vector \((S_{-N},\ldots,S_{-1})\). Expectations and stochastic comparisons in the deletion argument are conditional on this vector. For a vector \(\mathbf{v}=(v_{-n+1},\ldots,v_0)\) of dimension \(n\geq0\) with a unique maximizing suffix, define \(J(\mathbf{v}):=\argmax_{-n\leq m\leq0}\{\sum_{i=m+1}^{0}v_i\}\). It is obvious that \(b_1(0)=J(\mathbf{Y})\), where \(\mathbf{Y}=(Y_{-N+1},\ldots,Y_0)\).

The main idea of our proof is to construct a subset of \(\{-N,-N+1,\ldots,-1\}\), \(\mathcal R=\{i_1,i_2,\ldots,i_l\}\), with \(l=|\mathcal R|\), \(0\leq l\leq N\), and \(-N\leq i_1<\cdots<i_l\leq-1\), such that conditional on \(l\), \(\{Y_{i_j+1},1\leq j\leq l\}\) are IID with the same distribution as \(Y_i\), and
\begin{align}\label{eq_1}
 b_2(0)-b_1(b_2(0))\leq -J(\mathbf{Y}'),
\end{align}
where \(\mathbf{Y}'=(Y_{i_1+1},\ldots,Y_{i_l+1})\), with coordinates numbered from \(-l+1\) to \(0\).
Conditional on \(l\), adding \(N - l\) independent increments to the left of \(\mathbf{Y}'\) gives a vector with the law of \(\mathbf{Y}\). Every original suffix sum is unchanged, so this extension cannot decrease the length of the maximizing suffix. Averaging over \(l\) gives
\begin{align}\label{eq_2}
 -J(\mathbf{Y}')&\leq_{\rm st} -J(\mathbf{Y}),
\end{align}
where \(\leq _{st}\) denotes ``stochastically less than''. Putting (\ref{eq_1}) and (\ref{eq_2}) together leads to \(\E[b_2(0)-b_1(b_2(0))+b_1(0)]\leq\nobreak0\), which proves (\ref{eq_mainresult}).

To construct \(\mathcal R\), we use the following \textbf{deletion rule}. Define \(\mathcal K:=\{i\in\{-N,\ldots,-1\}:\Delta Z_{i+1}>0\}\). The position \(i\) corresponds to the increment \(Y_{i+1}\). Note that \(Z_i\geq Z_{i+1}\). This difference is determined only by the second service vector for fixed \(c\). 

\begin{enumerate}
\item \emph{Initialization.} Set $p:=-N$, $\mathcal R:=\emptyset$.
\item \emph{Find the leftmost undecided mark.} If $\mathcal K\cap\{p,\ldots,-1\}$ is empty, add all of $p,\ldots,-1$ to $\mathcal R$ and stop. If $p=0$, this set is empty. No value $Y_i$ is checked in this step. Otherwise, set
\(\beta:=\min(\mathcal K\cap\{p,\ldots,-1\})\).
Every index from \(p\) to \(\beta-1\)  is unmarked. Add those indices to \(\mathcal R\), again without checking their \(Y_{i+1}\) values. The next block will lie entirely at or to the right of \(\beta\).

\item \emph{Start the block.} Set \(\alpha:=\beta\), check \(Y_{\beta+1}\), and initialize the block sum by \(Q:=U_{\beta+1}\).
\item \emph{Extend right while the sum is positive.} If \(Q>0\) and \(\alpha<-1\), first replace \(\alpha\) by \(\alpha+1\), then check the newly included value \(Y_{\alpha+1}\) and add \(U_{\alpha+1}\) to \(Q\). Repeat this step while both conditions hold. At every stage, \(Q=\sum_{i=\beta+1}^{\alpha+1}U_i\).
Thus the decision to include the next position uses only the current sum. The new value is checked \emph{after} that position has been chosen.

\item \emph{Delete the block and continue or stop.} If \(Q\le0\), remove \([\beta,\alpha]\). Call this a \emph{completed block}. Replace \(p\) by \(\alpha+1\) and return to step 2. If \(Q>0\), the preceding step can have stopped only because \(\alpha=-1\). Remove \([\beta,-1]\) and stop; call this a \emph{positive leading block}. A sum of zero always belongs to the completed-block case.
\end{enumerate}

After a completed block \([\beta, \alpha]\), all later blocks lie in \(\alpha+1,\ldots,-1\). Thus blocks cannot overlap, and every check concerns a new position. The procedure stops after at most \(N\) checks. Each checked position belongs to a block that is then deleted. No position added to \(\mathcal R\) is checked. Conversely, all positions of a deleted block were checked in step 3 or step 4. At termination, \(\mathcal R\) is exactly the set of positions not deleted.

We note that an unmarked position is not automatically safe from deletion: it can be included when a block extends right in step 4. The positions skipped without checking are those to the \emph{left} of the current leftmost mark in step 2, or those left at the end when there is no mark. If a gap remains between two deleted blocks, this is how its positions were skipped.

Since 1) \(\mathcal K\) is entirely determined by \(\{S_i, -N\leq i\leq-1\}\), which are independent of \(\{Y_i, -N+1\leq i\leq 0\}\), and 2) \(\{Y_{i+1}, i\in \mathcal R\}\) are not checked during the construction of \(\mathcal R\), \(\{Y_{i+1}, i\in \mathcal R\}\) retain their original distributions, i.e., conditional on \(l\), they are IID with the same distribution as \(Y_i=c-A_i\). To verify this, condition on any complete sequence of checked indices: the event specifying this sequence depends only on the checked values, so independence leaves the complementary coordinates with their original product distribution. Summing over sequences with the same number of checks gives the assertion conditional on \(l\).

Next, we show \(b_2(0)-b_1(b_2(0))\leq -J(\mathbf{Y}')\). If \(b_2(0)-b_1(b_2(0))=0\), then \(b_2(0)-b_1(b_2(0))\leq -J(\mathbf{Y}')\) follows from \(-J(\mathbf{Y}')\geq 0\).  So we only need to consider \(b_2(0)-b_1(b_2(0))>0\). For a completed block \(I=[\beta, \alpha]\), define its sum by \(Q_I:=\sum_{i=\beta+1}^{\alpha+1}U_i\). Let \(\cal D\) be the collection of completed blocks, and let \(\beta_0\) be the left endpoint of the positive leading block; set \(\beta_0:=0\) if there is no such block. For \(I=[\beta,\alpha]\), write \(\min I:=\beta\). For each remaining position \(i\in\mathcal R\), define \(C(i):=-\sum_{I\in\mathcal D:\,\min I\geq i}Q_I\). Thus \(C(i)\) is the negative of the total of the completed blocks after \(i\). The possible positive leading block is not included in \(\mathcal D\) or \(C(i)\); its sum is \(\sum_{i=\beta_0+1}^{0}Y_i+Z_{\beta_0}=\sum_{i=\beta_0+1}^{0}U_i\), which is zero if \(\beta_0=0\).

\medskip\noindent\emph{Step 1. The position \(b_1(b_2(0))\) is not deleted.} We check the positions after a block's left endpoint and the endpoint itself separately. A block \([\beta,\alpha]=\{\beta,\ldots,\alpha\}\) includes both endpoints. For a completed block \(I=[\beta,\alpha]\), the stopping rule gives
\mbox{\(Q_I=\sum_{i=\beta+1}^{\alpha+1}U_i\leq0\)}.
For each \(\beta<i\leq \alpha\), the rule extended \([\beta,i-1]\) to include position \(i\). That extension was possible only because the sum over \([\beta,i-1]\) was positive. Hence
\(\sum_{j=\beta+1}^{i}U_j>0\). It follows that \(\sum_{j=i+1}^{0}U_j<\sum_{j=\beta+1}^{0}U_j\leq\sum_{j=b_1(b_2(0))+1}^{0}U_j\) for every \(i=\beta+1,\ldots,\alpha\), so none of these positions is \(b_1(b_2(0))\). At the remaining position \(i=\beta\), \(Z_i=\max\{Z_{i+1},\Xi_{i}+ic\}>Z_{i+1}\), which implies \(Z_i=\Xi_{i}+ic\) and \(i\) is its unique maximizing service index. If \(b_1(b_2(0))=i=\beta\), the maximizing pair \((b_1(b_2(0)),b_2(0))\) must therefore have \(b_2(0)=b_1(b_2(0))\), contradicting \(b_2(0)-b_1(b_2(0))>0\). Therefore, \(b_1(b_2(0))\neq\beta\).

For an existing positive leading block \([\beta_0,-1]\), every sum checked during its formation was positive, including its full sum. Therefore \(\sum_{j=\beta_0+1}^{i}U_j>0\) for \(\beta_0<i\leq-1\), which gives \(\sum_{j=i+1}^{0}U_j<\sum_{j=\beta_0+1}^{0}U_j\leq\sum_{j=b_1(b_2(0))+1}^{0}U_j\), so \(i\ne b_1(b_2(0))\). We can similarly show \(\beta_0\ne b_1(b_2(0))\). Thus the leading block does not delete \(b_1(b_2(0))\) either, despite having a positive total. We conclude that \(b_1(b_2(0))\in\mathcal R\).

\medskip\noindent\emph{Step 2. No position from \(b_1(b_2(0))\) to \(b_2(0)-1\) is deleted.} Define \(\mathcal A:=\{b_1(b_2(0)),\ldots,b_2(0)-1\}\).
The unique maximizing pair makes \(b_2(0)\) the unique index attaining \(Z_{b_1(b_2(0))}\). Thus \(Z_i=Z_{i+1}\) for \(b_1(b_2(0))\leq i<b_2(0)\), so no position in \(\mathcal A\) is marked. Suppose a deleted block \([\beta,\alpha]\) intersects \(\mathcal A\). Its left endpoint \(\beta\) is marked and hence cannot belong to \(\mathcal A\). An intersection therefore requires \(\beta<b_1(b_2(0))\leq\alpha\), which puts \(b_1(b_2(0))\) in the deleted block, contradicting step 1. Therefore \(\mathcal A\subseteq\mathcal R\).

\medskip\noindent\emph{Step 3. Compute the partial sums after deletion.} For \(i\in\mathcal R\), let \(\ell(i):=-|\mathcal R\cap\{i,\ldots,-1\}|\), the corresponding starting index after deletion. Since all \(b_2(0)-b_1(b_2(0))\) positions in \(\mathcal A\) remain and are at least \(b_1(b_2(0))\), we have \(-\ell(b_1(b_2(0)))\geq b_2(0)-b_1(b_2(0))\). We should emphasize the fact that the choice of \(b_1(b_2(0))\) is only part of the proof; the deletion rule does not search for it or check the values at positions in \(\mathcal A\). For a remaining position \(i\in\mathcal R\), we have \(\sum_{j=i+1}^{0}U_j=\sum_{j\in\mathcal R,j\geq i}U_{j+1}+\sum_{j=\beta_0+1}^{0}U_j+\sum_{I\in\cal D,\min I\geq i}Q_I\), which can be rewritten as \(\sum_{j\in\mathcal R,j\geq i}U_{j+1}=\sum_{j=i+1}^{0}U_j-\sum_{j=\beta_0+1}^{0}U_j+C(i)\).
Since every completed-block total satisfies $Q_I\le0$, for $i,j\in\mathcal R$ with $i\geq j$ we have
\begin{align}\label{eq_correction-order}
 C\left(i\right)&\geq0,\qquad
 C\left(j\right)-C\left(i\right)=-\sum_{\substack{I\in\cal D\\j\leq\min I<i}}Q_I\geq0.
\end{align}
\Needspace{6\baselineskip}
\medskip\noindent\emph{Step 4. Prove \(J(\mathbf{Y}')\leq\ell(b_1(b_2(0)))\).}
For any remaining position \(i\geq b_1(b_2(0))\), we have
\begin{align*}
 \sum_{\substack{j\in\mathcal R\\j\geq b_1(b_2(0))}}U_{j+1}
 -\sum_{\substack{j\in\mathcal R\\j\geq i}}U_{j+1}
 &=\sum_{j=b_1(b_2(0))+1}^{0}U_j-\sum_{j=i+1}^{0}U_j+C\left(b_1(b_2(0))\right)-C\left(i\right)\geq0.
\end{align*}
The first difference is nonnegative because \(\sum_{j=b_1(b_2(0))+1}^{0}U_j\) is largest among the original partial sums; the second is nonnegative by \eqref{eq_correction-order}. Since \(\beta_0\in\{-N,\ldots,0\}\), we also have \(\sum_{j=b_1(b_2(0))+1}^{0}U_j\geq\sum_{j=\beta_0+1}^{0}U_j\), and hence \(\sum_{j\in\mathcal R,j\geq b_1(b_2(0))}U_{j+1}=\sum_{j=b_1(b_2(0))+1}^{0}U_j-\sum_{j=\beta_0+1}^{0}U_j+C(b_1(b_2(0)))\geq0\).
The retained entries \(Y_{j+1}\) with \(j\geq b_1(b_2(0))\) form a suffix of \(\mathbf{Y}'\). Since \(U_{j+1}=Y_{j+1}\) at retained positions, its sum is nonnegative and no smaller than any shorter suffix sum. Hence \(J(\mathbf{Y}')\leq\ell(b_1(b_2(0)))\), and \(-J(\mathbf{Y}')\geq-\ell(b_1(b_2(0)))\geq b_2(0)-b_1(b_2(0))\).

With the above four steps, we have completed our proof for Theorem \ref{thm:mean}.

\section{A special case: \(M/D/1 \rightarrow M/1\)}
In this section, we give a different proof of Theorem \ref{thm:mean} for the case in which both \(A\) and \(S\) are exponentially distributed with rates \(\lambda\) and \(\mu\), respectively. The basic idea of our approach is to try to calculate \(\E[b_2(0)-b_1(b_2(0))+b_1(0)]\) more explicitly. Throughout this section, we fix \(0<\lambda<\mu\), take \(0<c<1/\lambda\), and assume the system is in steady state. All the proofs for lemmas in this section are provided in the appendix.

A \emph{departure group} from station 1 consists of the customers served during one busy period at the station. Let \(B\) be the size of a busy period, and let \(\nu_n=\Prob\{B=n\}\). For customer \(k\), write \(B_k\) for the size of its group, \(R_k\) for its position counted from the beginning, and \(H_k=B_k-R_k+1\) for its position counted from the end. When the dependence on \(c\) must be displayed, we write \(B_k(c),R_k(c),H_k(c)\); otherwise it is suppressed. Since the first queue is an \(M/D/1\) queue, we have the following lemma (\(\rho=\lambda c\)). 

\begin{lemma}[Reversible departure sequence]\label{lem:reversal}
Within each group, departures are spaced by \(c\). Under group indexing, between successive groups the gap is \(c+E\), where the additional gaps \(E\) are independent exponential variables of rate \(\lambda\), independent of the group sizes. The variable \(B\) has an exponential tail,
\begin{equation}\label{eq_group_size}
 \E[B]=\frac{1}{1-\rho},\qquad \nu_2=\rho e^{-2\rho}>0,
\end{equation}
and, at a stationary customer,
\begin{equation}\label{eq_group_palm}
 \Prob\{B_0=n,R_0=r\}=\frac{\nu_n}{\E[B]},\qquad 1\le r\le n.
\end{equation}
For each fixed \(c\), the marked stationary departure process satisfies
\begin{equation}\label{eq_marked_reversal}
 \big(D_{1,k}-D_{1,0},R_k,H_k,S_k\big)_{k\in\Z}
 \eqdist
 \big(D_{1,0}-D_{1,-k},H_{-k},R_{-k},S_{-k}\big)_{k\in\Z}.
\end{equation}
Thus the process is invariant in distribution under customer-index reversal, with the beginning and end position marks exchanged. In particular, \(\E[R_0]=\E[H_0]<\infty\).
\end{lemma}

Recall the index \(b_2(0)\) is the first customer in customer 0's busy period at station 2 and \(b_1(b_2(0))\) is the first customer of the busy period containing \(b_2(0)\) at station 1; therefore, \(b_2(0)-b_1(b_2(0))=R_{b_2(0)}-1\).  The same argument at station 1 gives \(-b_1(0)=R_0-1\).  Then together with the reversibility result in Lemma~\ref{lem:reversal}, we have
\begin{lemma}[Centered derivative]\label{lem:centered}
\begin{align}\label{eq_centered}
 \frac{d}{dc}\E[W_{2,0}(c)]
 &=\E[b_2(0)-b_1(b_2(0))+b_1(0)]
 =\E[R_{b_2(0)}-R_0] \nonumber\\
 &=\E[H_0-H_{b_2(0)}]
 =\frac{\E[R_{b_2(0)}-H_{b_2(0)}]}{2}.
\end{align}
\end{lemma}

We next analyze the second station for the fixed \(c\). Its population immediately after an arrival takes values in \(\{1,2,\ldots\}\). Exponential service can be represented by an independent Poisson clock of rate \(\mu\), whose points remove a customer whenever the queue is nonempty. Let \(M(t)\) denote its number of points over an interval of length \(t\), and let \(E\) be independent exponential of rate \(\lambda\).

Let \(P\) be the transition kernel across an interval \(c\), including the arrival at its end, and let \(V\) be the corresponding kernel across \(c+E\). Thus, for a bounded function \(f\), \((Pf)(q)=\E[f(1+(q-M(c))^+)]\) and \((Vf)(q)=\E[f(1+(q-M(c+E))^+)]\). Kernel products act in chronological order. Stochastic order between kernels means pointwise comparison after application to every bounded increasing function. Both \(P\) and \(V\) preserve increasing functions.

We also need the substochastic kernels \(\widetilde P\) and \(\widetilde V\) that stop when the queue first empties. For a nonnegative function \(f\) on \(\{1,2,\ldots\}\), define \((\widetilde P f)(q)=\E[f(1+(q-M(c))^+)\ind\{M(c)<q\}]\) and \((\widetilde V f)(q)=\E[f(1+(q-M(c+E))^+)\ind\{M(c+E)<q\}]\). On the survival event, these function arguments equal \(q-M(c)+1\) and \(q-M(c+E)+1\), respectively; on all other paths the contribution is zero. An inequality between these substochastic kernels below means pointwise inequality after application to every nonnegative function.

The distinction between a population observed after two arrivals and a busy period surviving to those arrivals is important.
\begin{lemma}[Interchanging a short and a long gap]\label{lem:kernels}
The kernels satisfy \(PV\le_{\rm st}VP,~\widetilde P\widetilde V\ge\widetilde V\widetilde P,~\widetilde P\boldsymbol{1}\ge\widetilde V\boldsymbol{1}\), where \(\boldsymbol{1}\) is the constant-one function. At state~1, we have
\begin{align}\label{eq_strict_survival}
 (\widetilde P\boldsymbol{1}-\widetilde V\boldsymbol{1})(1)
 &=\frac{e^{-\mu c}\mu}{\lambda+\mu}>0.
\end{align}
\end{lemma}

Index departure groups from station 1 by \(g\in\Z\), with a typical group start as the origin, let \(\mathcal Q_g\) be the population at station 2 immediately after the first arrival of group \(g\), and let \(\pi\) be its stationary distribution. The state \(\mathcal Q_g\) depends on previous departure groups, the preceding additional gap, and downstream service, and is independent of the current group size. Independence of successive group sizes therefore gives the cycle kernel
\begin{equation}\label{eq_cycle_kernel}
 K=\sum\nolimits_{n\ge1}\nu_n P^{n-1}V,\qquad \pi K=\pi.
\end{equation}
For \(r\ge1\), let \(p_r=(\pi P^{r-1})(\{1\})\), the probability of zero downstream waiting at position \(r\), conditional on the group containing that position. Conditioning further on any particular group size \(n\ge r\) leaves this probability unchanged. For \(m\ge0\) and \(q\ge1\), let \(t_m(q)\) be the expected number of current and subsequent arrivals counted before the downstream queue next empties, when its current population is \(q\) and exactly \(m\) further arrivals remain in the current upstream group. Future additional gaps and groups have their usual independent laws. In particular, \(t_m(1)\) is the expected size of a downstream busy period starting at this position.

\begin{lemma}[Position within a departure group]\label{lem:position}
The sequence \(p_r\) is nonincreasing and strictly positive. The sequence \(t_m(q)\) is finite and nondecreasing in \(m\) for every \(q\), and
\begin{align}\label{eq_b_strict}
 t_1(1)-t_0(1)&\geq\frac{e^{-\mu c}\mu}{\lambda+\mu}>0.
\end{align}
\end{lemma}

We are now ready to prove Theorem \ref{thm:mean} by showing \(\E[b_2(0)-b_1(b_2(0))+b_1(0)]<0\).
Let \(L_k\) be the number of customers in the busy period starting with customer \(k\) at station 2, with \(L_k=0\) when \(k\) does not start one. For any nonnegative stationary customer mark \(f_k\), stationarity under customer-index shifts gives
\begin{align*}
 \E[f_{b_2(0)}]
 &=\sum\nolimits_{j\leq0}\E[f_j\ind\{b_2(0)=j\}]
 =\sum\nolimits_{k\geq0}\E[f_0\ind\{b_2(k)=0\}]
 =\E[f_0L_0].
\end{align*}
The middle equality shifts the event \(b_2(0)=j\) by \(-j\), so it becomes \(b_2(-j)=0\). The exponential bound for \(t_m\) in the proof of Lemma~\ref{lem:position}, together with the exponential tail of \(B\), gives \(\E[B_0L_0]<\infty\) by conditioning on \((B_0,R_0)\) as below. Applying the identity first to \(|R_k-H_k|\), and then to the positive and negative parts of \(R_k-H_k\), therefore yields
\begin{equation}\label{eq_mass_transport}
 \E[R_{b_2(0)}-H_{b_2(0)}]=\E[(R_0-H_0)L_0].
\end{equation}

We next compute the conditional mean of \(L_0\). Fix \(B_0=n\) and \(R_0=r\). The state at station 2 seen by customer 0 is determined by the state at the first arrival of its group from station 1, the first \(r-1\) within group gaps, and the associated downstream service completions. Conditional on the group reaching position \(r\), those gaps are exactly \(c\), regardless of the eventual value of \(B_0\). Moreover, the current group size is determined by the station-1 Poisson arrivals after the group begins, while the downstream service process is independent, and the downstream state at the first arrival of the group depends only on the past. Consequently, conditioning further on \(B_0=n\) does not change the probability that customer~0 finds station~2 empty, and this probability is \(p_r\).

If customer~0 does not find the second station empty, then \(L_0=0\). If it does, it starts a downstream busy period with postarrival population one. Given \(B_0=n,R_0=r\), exactly \(n-r\) further arrivals remain in the current group, all separated by gaps \(c\). After that group ends, the boundary idle time, all later group sizes and boundary gaps, and all future downstream service clock increments are independent of the past and have the laws used to define \(t_{n-r}(1)\). Hence \(\E[L_0\mid B_0=n,R_0=r]=p_r t_{n-r}(1)\).

Because \(H_0=B_0-R_0+1\), on the event \(B_0=n,R_0=r\) we have \(R_0-H_0=2r-n-1\). Combining this fact with \eqref{eq_group_palm}, \eqref{eq_centered}, and \eqref{eq_mass_transport} gives the absolutely convergent identity
\begin{align}\label{eq_group_derivative}
 \E[b_2(0)-b_1(b_2(0))+b_1(0)]
 &=\frac{1}{2\E[B]}\sum\nolimits_{n\geq1}\nu_n\sum\nolimits_{r=1}^{n}(2r-n-1)p_rt_{n-r}(1).
\end{align}
Fix \(n\). Lemma~\ref{lem:position} shows that \(p_r t_{n-r}(1)\) is nonincreasing in \(r\): the first factor is nonincreasing, and the second is also nonincreasing because \(t_m(1)\) is nondecreasing in \(m\) while \(n-r\) decreases. Pairing positions equidistant from the two ends of the group gives \(\sum\nolimits_{r=1}^{n}(2r-n-1)p_r t_{n-r}(1)=\sum\nolimits_{r=1}^{\lfloor n/2\rfloor}(n+1-2r)\{p_{n+1-r}t_{r-1}(1)-p_rt_{n-r}(1)\}\le0\). Thus every group size contributes a nonpositive term to \eqref{eq_group_derivative}. For \(n=2\), the inner sum is \(-p_1t_1(1)+p_2t_0(1)\), which is strictly negative because \(p_2\le p_1\), \(p_1>0\), and \(t_1(1)>t_0(1)\). Since \(\nu_2>0\) by \eqref{eq_group_size}, we conclude that \(\E[b_2(0)-b_1(b_2(0))+b_1(0)]<0\).

\section{Conclusion}\label{sec:conclusion}
In this paper, we prove that Niu's conjecture is true, solving a long standing open problem in queueing theory.  We provide two different proofs: one for the general case and one for a special case.  

\renewcommand{\theHsection}{appendix.\Alph{section}}
\Needspace{6\baselineskip}
\begin{APPENDIX}{}

\proof{Proof of Lemma~\ref{lem:reversal}.}
\label{app:global_proofs}
Since station~1 is an \(M/D/1\) queue, \(B\) has a Borel distribution with parameter \(\rho\), which immediately gives (\ref{eq_group_size}).  Furthermore, a busy period \(g\in\Z\) is followed by an exponentially distributed idle time \(E_g\), and then by the next busy period \(g+1\).  Writing \(B^{(g)}\) for the size of group \(g\) under this construction, both \(\{B^{(g)},g\in\Z\}\) and \(\{E_g, g\in\Z\}\) are IID and independent of each other.   

Now consider stationary customer~0.  In equilibrium, the busy period containing customer~0 is size biased, so \(\Prob\{B_0=n\}=n\nu_n/\E[B]\). Conditional on \(B_0=n\), customer~0 is uniformly located among the \(n\) positions in that group, so \(\Prob\{R_0=r\mid B_0=n\}=1/n\) for \(1\le r\le n\). Groups on either side of this central group retain the ordinary IID size law, and the additional gaps remain IID and independent of all group sizes and the central position. Therefore \(\Prob\{B_0=n,R_0=r\}=\Prob\{R_0=r\mid B_0=n\}\Prob\{B_0=n\}=\nu_n/\E[B]\), which proves \eqref{eq_group_palm}.

It remains to verify reversal of the boundary marks. In forward customer order, group \(g\) contributes \(B^{(g)}-1\) gaps equal to \(c\), followed by the boundary gap \(c+E_g\). After reversing the customer order and regrouping at the gaps larger than \(c\), this group is followed by \(c+E_{g-1}\), not by \(c+E_g\). In the group-start construction, reversal therefore transforms \(\{(B^{(g)},E_g)\}_{g\in\Z}\) into \(\{(B^{(g)},E_{g-1})\}_{g\in\Z}\), up to reversal of the group index. These block sequences have the same law because the size and gap sequences are independent and each is IID. Under equilibrium customer indexing, the central group retains its size-biased size, the ordinary groups on its two sides are interchanged, and the central position \(R_0=r\) becomes \(H_0=n-r+1\) when \(B_0=n\). The conditional uniformity of that position and independence of all additional gaps therefore give invariance of the marked departure sequence under reversal. Since \(\{S_k\}_{k\in\Z}\) is IID and independent of the first-station process, the invariance extends jointly to the downstream service marks. This proves the marked distributional identity \eqref{eq_marked_reversal}. In particular, \(R_0\) and \(H_0\) have the same distribution. Their means are finite because the group size seen by a stationary customer is size biased from an exponentially tailed \(B\).
\Halmos\endproof

\proof{Proof of Lemma~\ref{lem:centered}.}
By Lemma~\ref{lem:reversal} and independence of the service sequence,
\begin{align*}
 \E[W_{2,0}(c)]
 &=\E\big[\max\nolimits_{j\geq0}\big\{\sum\nolimits_{i=0}^{j-1}S_i-D_{1,j}(c)+D_{1,0}(c)\big\}\big].
\end{align*}

To apply the IPA interchange, fix an upper bound \(0<c_+<1/\lambda\). For \(0\leq c\leq c_+\) and \(j\geq1\),
\begin{align*}
 \Xi_{-j}-D_{1,0}(c)+D_{1,-j}(c)
 &\leq\Xi_{-j}+\tau_{-j}+W_{1,-j}(c_+),\\
 \sum\nolimits_{i=0}^{j-1}S_i-D_{1,j}(c)+D_{1,0}(c)
 &\leq\sum\nolimits_{i=0}^{j-1}S_i-\tau_j+W_{1,0}(c_+).
\end{align*}
Both upper bounds have exponential moment \(\E[e^{\theta W_{1,0}(c_+)}](\E[e^{\theta S}]\E[e^{-\theta A}])^j\) for sufficiently small \(\theta>0\), with a finite first factor and a base less than one. A Chernoff bound therefore bounds both maximizing indices, uniformly on \([0,c_+]\), by a random integer with an exponential tail. Since \(dD_{1,k}(c)/dc=R_k(c)\leq R_k(c_+)\), the exponential tail of the first-station busy periods and Cauchy--Schwarz give an integrable bound on the difference quotients. The standard IPA interchange (\citealt{hu1992,ghs1991}) yields
\begin{align*}
 \frac{d}{dc}\E[W_{2,0}(c)]
 &=\E[R_{b_2(0)}-R_0]
 =\E[H_0-H_{b_2(0)}].
\end{align*}
Here the last expression follows by reversing the differentiated forward maximum with service marks \(S_{-k-1}\). Its maximizing customer becomes \(b_2(0)\), and beginning and end position marks are interchanged. Averaging the two position expressions and using \(\E[R_0]=\E[H_0]\) proves \eqref{eq_centered}.
\Halmos\endproof

\proof{Proof of Lemma~\ref{lem:kernels}.}
Condition on \(E=e\) and couple two systems by the same rate-\(\mu\) service clock on \([0,2c+e]\), starting from the same population \(q\). In the \(PV\) system the next two arrivals occur after gaps \(c\) and \(c+e\), so the intermediate arrival is at time \(c\). In the \(VP\) system the gaps are reversed, and the intermediate arrival is at time \(c+e\). Both systems receive the terminal arrival at time \(2c+e\).

Before time \(c\) the systems coincide. At time \(c\), \(PV\) receives one extra customer. Between \(c\) and \(c+e\), this extra customer can only prevent a service clock point from being wasted; it cannot make the \(PV\) population more than one larger than the \(VP\) population. When the \(VP\) system receives its intermediate arrival at time \(c+e\), its population becomes at least the \(PV\) population. From that time onward both systems see the same service clock and the same terminal arrival, so this order is preserved. Hence the terminal population under \(PV\) is no larger pathwise, proving \(PV\le_{\rm st}VP\).

For the stopped kernels, suppose instead that the \(VP\) busy period survives until the terminal arrival. Then it has not emptied during the longer first gap \([0,c+e]\). The two systems agree up to time \(c\), and the earlier intermediate arrival in \(PV\) can only help it remain nonempty on \([c,c+e]\). Thus survival of \(VP\) implies survival of \(PV\). On this common survival event no service opportunity has been wasted in either system; immediately after the \(VP\) intermediate arrival at time \(c+e\), both populations are therefore equal, and they remain equal until the terminal arrival. Consequently, for every nonnegative \(f\), every contribution to \(\widetilde V\widetilde P f\) also appears in \(\widetilde P\widetilde V f\) with the same terminal state, while \(PV\) may survive on additional paths. This proves \(\widetilde P\widetilde V\ge\widetilde V\widetilde P\). The inequality \(\widetilde P\boldsymbol{1}\ge\widetilde V\boldsymbol{1}\) follows because survival over a gap of length \(c+E\) implies survival over its initial subinterval of length \(c\). From state~1 these survival probabilities are \(e^{-\mu c}\) and \(e^{-\mu c}\E[e^{-\mu E}]\); using \(\E[e^{-\mu E}]=\lambda/(\lambda+\mu)\) gives \eqref{eq_strict_survival}.
\Halmos\endproof

\proof{Proof of Lemma~\ref{lem:position}.}
We first prove the monotonicity of \(p_r\). From \(PV\le_{\rm st}VP\) and \eqref{eq_cycle_kernel}, \(PK=\sum\nolimits_{n\ge1}\nu_nP^nV\le_{\rm st}\sum\nolimits_{n\ge1}\nu_nP^{n-1}VP=KP\). The kernels \(P\), \(V\), and \(K\) preserve stochastic order, so iteration gives \(\delta_1K^jP\ge_{\rm st}\delta_1PK^j\ge_{\rm st}\delta_1K^j\) for \(j\ge1\); the second inequality uses \(\delta_1P\ge_{\rm st}\delta_1\), since state~1 is the smallest possible postarrival population.

To pass to the limit, recall from the definition of \(\pi\) that \(\mathcal Q_g\) is the downstream population immediately after the first arrival of upstream group \(g\). Under the stationary group construction, the next group size, its boundary gap, and the future service clock increments are independent of the past, so \((\mathcal Q_g)\) is a stationary Markov chain with kernel \(K\) and marginal law \(\pi\); hence \(\pi K=\pi\). The chain is irreducible: from any state, a one-customer group followed by sufficiently many service completions during the boundary gap reaches state~1 with positive probability, while from state~1 repeated one-customer groups with no service completion can reach any larger state. Since \(K(1,1)>0\), the chain is also aperiodic. An irreducible countable state chain with invariant probability \(\pi\) is positive recurrent, so \(\delta_1K^j\to\pi\). Taking limits in the preceding stochastic order gives \(\pi P\ge_{\rm st}\pi\), and applying further powers of \(P\) shows that \(\pi P^{r-1}\) is nondecreasing in stochastic order. Its mass at the minimum state must therefore be nonincreasing, which is exactly \(p_{r+1}\le p_r\). Irreducibility gives \(p_1>0\), and \(P(q,\{1\})=\Prob\{M(c)\ge q\}>0\) implies \(p_r>0\) for every \(r\).

We next compare the busy-period counts. Define the affine operators \(\mathcal P f=\boldsymbol{1}+\widetilde P f\) and \(\mathcal V f=\boldsymbol{1}+\widetilde V f\). The constant term counts the current arrival; the second term continues the count only when the busy period survives to the next arrival. Lemma~\ref{lem:kernels} gives, for every bounded nonnegative \(f\),
\begin{align}\label{eq_affine_order}
 \mathcal P\mathcal V f-\mathcal V\mathcal P f
 &=\widetilde P\boldsymbol{1}-\widetilde V\boldsymbol{1}+(\widetilde P\widetilde V-\widetilde V\widetilde P)f\ge0.
\end{align}
Applying this identity to bounded truncations and then using monotone convergence extends the comparison \(\mathcal P\mathcal V f\ge\mathcal V\mathcal P f\) to all nonnegative \(f\).
Let \(h(q)\) be the expected count starting with the first arrival of a fresh upstream group, and define \(\mathcal G f=\sum\nolimits_{n\ge1}\nu_n{\mathcal P}^{n-1}\mathcal V f\). Processing only the first \(j\) groups gives \(\mathcal G^j0\), so monotone convergence yields \(h=\lim_{j\to\infty}\mathcal G^j0\). If exactly \(m\) further arrivals remain in the current group, the same counting rule gives \(t_m={\mathcal P}^m\mathcal Vh\). For sufficiently small \(\theta>0\), let \(\kappa=\E[\exp\{\theta B+\mu(cB+E)(e^{-\theta}-1)\}]<1\), where the strict inequality follows because the derivative at zero is \(\E[B](1-\mu/\lambda)<0\). On survival through a group the next population is \(q+B-M(cB+E)\), so the expected contribution of group \(j+1\) is at most \(\E[B]\kappa^j e^{\theta q}\). Summing over \(j\geq0\) gives \(h(q)\leq\E[B]e^{\theta q}/(1-\kappa)\) and \(t_m(q)\leq m+1+\E[B]e^{\theta(q+m+1)}/(1-\kappa)<\infty\). The value of \(\theta\) can be chosen arbitrarily small, which also gives the integrability used in \eqref{eq_mass_transport}. 

Equation~\eqref{eq_affine_order} allows one \(\mathcal P\) to be moved past \(\mathcal V\) to the right, and hence \(\mathcal P\mathcal G f\ge\mathcal G\mathcal P f\) for nonnegative \(f\). Iterating and using \(\mathcal P0=\boldsymbol{1}\ge0\) gives \(\mathcal P\mathcal G^j0\ge\mathcal G^j\mathcal P0\ge\mathcal G^j0\). Letting \(j\to\infty\) and using monotone convergence in the substochastic kernels yields \(\mathcal Ph\ge h\). Therefore \(\mathcal P\mathcal Vh\ge\mathcal V\mathcal Ph\ge\mathcal Vh\), and applying the order preserving operator \({\mathcal P}^m\) gives \(t_{m+1}\ge t_m\). Finally, at state~1, \eqref{eq_affine_order} and \eqref{eq_strict_survival} yield \(t_1(1)-t_0(1)\ge(\widetilde P\boldsymbol{1}-\widetilde V\boldsymbol{1})(1)>0\), proving \eqref{eq_b_strict}.
\Halmos\endproof

\end{APPENDIX}

\bibliographystyle{informs2014}
\bibliography{references}

\end{document}